\documentclass[12pt]{article}
\usepackage{amsmath}
\usepackage{amsthm}
\usepackage{amssymb}
\usepackage{tikz}
\usepackage{amsfonts}
\usepackage{epsf}
\usepackage{graphicx}
\usepackage{epstopdf}
\usepackage{hyperref}
\usepackage{appendix}
\usepackage{graphics}
\usepackage{graphicx}
\usepackage{CJK}
\newtheorem{lemma}{Lemma}

\newtheorem{theorem}{Theorem}

\newtheorem{corollary}{Corollary}

\theoremstyle{remark}

\newtheorem{remark}{Remark}
\theoremstyle{definition}

\DeclareMathOperator\re{{Re}}
\DeclareMathOperator\im{{Im}}
\numberwithin{equation}{section}

\newcommand{\HM}{\mathrm{HM}}

\numberwithin{equation}{section}

\newcounter{comment}
\begin{document}
\title{Asymptotics and total integrals of the $\mathrm{P}_{\rm I}^{2}$ tritronqu\'{e}e solution and its Hamiltonian}

\author{Dan Dai\footnotemark[1] \ and Wen-Gao Long\footnotemark[2]}

\renewcommand{\thefootnote}{\fnsymbol{footnote}}
\footnotetext[1]{Department of Mathematics, City University of Hong Kong, Tat Chee
Avenue, Kowloon, Hong Kong.\\ E-mail: \texttt{dandai@cityu.edu.hk}}
\footnotetext[2]{School of Mathematics and Computational Science, Hunan University of Science and Technology, Xiangtan, Hunan, China. \\
E-mail: \texttt{longwg@hnust.edu.cn} (corresponding author)} 

\date{\today}

\maketitle

\begin{abstract}
We study the tritronqu\'{e}e solution $u(x,t)$ of the $\mathrm{P}_{\rm I}^{2}$ equation, the second member of the Painlev\'{e} I hierarchy. This solution is pole-free on the real line and has various applications in mathematical physics. We obtain a full asymptotic expansion of $u(x,t)$ as $x\to\pm \infty$, uniformly for the parameter $t$ in a large interval. Based on this result, we successfully derive the total integrals of $u(x,t)$ and the associated Hamiltonian.
\end{abstract}


\vspace{5mm}



\section{Introduction}
The second member in the Painlev\'{e} I hierarchy, also referred as the $\textrm{P}_{\textrm I}^2$ equation, is the following fourth-order ordinary differential equation
\begin{equation}\label{eq-PI-hierarchy}
u_{xxxx}+10u_{x}^2+20uu_{xx}+40(u^3-6tu+6x)=0 
\end{equation}
with a parameter $t \in \mathbb{R}.$ Like the Painlev\'e I equation
\begin{equation}\label{eq-PI-def}
u_{xx} = 6u^2 + x,
\end{equation}
general solutions $u(x,t)$ of $\textrm{P}_{\textrm I}^2$ are meromorphic in $x$, and have infinitely many poles in the complex-$x$ plane. However, researchers have identified a class of pole-free solutions $u(x,t)$ for $x$ on the real axis. These pole-free solutions play a significant role in various topics of mathematical physics.

In the study of string theory, Br\'{e}zin et al. \cite{Bre-Mar-Par} and Moore \cite{Moore} discovered that, the equation \eqref{eq-PI-hierarchy} (with $t=0$) possesses a special solution which is pole-free on the real line and satisfies the following asymptotic behavior:
\begin{equation} \label{eq-p1-BC}
u(x,t=0) \sim \mp|6x|^{\frac{1}{3}}, \qquad \text{as }  x\to\pm\infty.
\end{equation}
Later, uniqueness of the real pole-free solution $u(x,t=0)$ was proved in Kapaev \cite{Kapaev1995}. For general $t\neq 0$, Dubrovin \cite{Dubrovin-2006} conjectured that there exists a unique real pole-free solution $u(x,t)$ as well. The existence of the real pole-free solution $u(x,t)$ to $\textrm{P}_{\textrm I}^2$ was proved in Claeys and Vanlessen \cite{Claeys-Vanlessen-2007} for any $t \in \mathbb{R}$. The main motivation to study this solution in \cite{Dubrovin-2006} is that, Dubrovin suggested the particular solution $u(x,t)$ describes the universal asymptotics for Hamiltonian perturbations of hyperbolic equations near the point of gradient catastrophe for the unperturbed equation. Nowadays, this is known as the \emph{universality conjecture} for Hamiltonian PDEs; see also \cite{Dubrovin-Grava-Klein-Moro-2015}. As far as we know, the conjecture has only been verified for  the Korteweg-de Vries equation (KdV)
\begin{equation}\label{eq-kdv}
u_{t}+uu_{x}+\frac{1}{12}u_{xxx}=0
\end{equation}
and its hierarchy in \cite{Claeys-Grava-2009,Claeys-Grava-2012}. Recently, in the study of two-dimensional dispersive shock waves, Dubrovin et al. \cite{Dubrovin-Grava-Klein-2016} made another conjecture that, the asymptotic description of solutions to the generalized Kadomtsev-Petviashvili (KP) equations can also be given in terms of the real pole-free solution of $\textrm{P}_{\textrm I}^2$. This solution is also found to describe certain critical behavior in nonlinear waves; see \cite{Mas:Rai:Ant2015}. It is interesting to note that, if $u(x,t)$ is a solution of $\textrm{P}_{\textrm I}^2$, it is a solution of the KdV equation as well. Then, in a different context, the particular pole-free solution $u(x,t)$ is referred as the Gurevich-Pitaevskii solution of the KdV equation; cf. \cite{Suleimanov1994}.

The real pole-free solution $u(x,t)$ of $\textrm{P}_{\textrm I}^2$ also appears in random matrix theory and orthogonal polynomials. It is well-known that, local eigenvalue statistics of various random matrix ensembles exhibit universality properties when the matrix size $n$ tends to infinity; see \cite{Kuijlaars-Survey} and references therein. For the random unitary ensembles, the limiting eigenvalue correlation kernels are expressed in terms of the sine function and the Airy function, in the bulk or at the soft edge of the spectrum, respectively. If the limiting mean eigenvalue density vanishes like a power 5/2 near a singular edge point, the local eigenvalues correlation kernel is given in terms of functions associated with the real pole-free solutions of $\textrm{P}_{\textrm I}^2$; see Claeys and Vanlessen \cite{Claeys-Vanlessen-2007-2}. Moreover, the real pole-free solution $u(x,t)$ appears explicitly in the second term of the asymptotic expansions for recurrence coefficients of the associated orthogonal polynomials; cf. \cite{Claeys-Vanlessen-2007-2}.

Due to the importance of the real pole-free solution, there has been considerable interest in studying its properties. For example, besides the existence, it is also shown in \cite{Claeys-Vanlessen-2007} that $u(x,t)$ satisfies the following asymptotic behavior:
\begin{equation} \label{eq-p1-t-BC}
u(x,t)=\frac{1}{2} z_0 |x|^{\frac{1}{3}} + O(|x|^{-2}), \qquad \text{as }  x\to\pm\infty,
\end{equation}
for fixed $t$, where $z_0$ is the real solution of 
$$
z_0^3 = -48 \, \textrm{sgn}(x) + 24 z_0 |x|^{-2/3} t.
$$ 
In \cite{Claeys-2010}, Claeys extended the above asymptotic results by considering $u(x,t)$ when $x$ and $t$ tend to infinity simultaneously. Depending on the precise scaling of $x$ and $t$, both algebraic and elliptic asymptotics were derived. Recently, Grava et al. \cite{Grava-Kapaev-Klein-2015} studied properties of $u(x,t)$ in the complex-$x$ plane. They showed that, $u(x,t)$ is not only real and regular on the real line, but also has an extension to the complex plane with uniform algebraic asymptotics in large sectors. As $u(x,t)$ shares some similar features as the well-known \emph{tritronqu\'{e}e solution} to the Painlev\'e I equation \eqref{eq-PI-def}, they call 
$u(x,t)$ the tritronqu\'{e}e solution of $\textrm{P}_{\textrm I}^2$ and conjectured it to be pole-free in the large sectors in the complex-$x$ plane; see the precise statement of the conjecture in \cite[Conjecture 1.1]{Grava-Kapaev-Klein-2015}. For readers who are interested in other members in the Painlev\'{e} I hierarchy, we refer to \cite{Claeys-2012,Dai-Zhang-2010}.

In this paper, we intend to deepen our understanding about the tritronqu\'{e}e solution $u(x,t)$ of the $\textrm{P}_{\textrm I}^2$ equation by further investigating its asymptotics, as well as the associated Hamiltonians $H_1(x,t)$ and $H_2(x,t)$. These two Hamiltonians are defined explicitly as follows:
\begin{align}
H_{1}(x,t)&=xu+\frac{1}{24}u^{4}-\frac{1}{2}tu^2+\frac{1}{24}uu_{x}^2+\frac{1}{240}u_{x}u_{xxx}-\frac{1}{480}u_{xx}^{2},  \label{eq-represent-Hamiltonian-1} \\
H_{2}(x,t)&=\frac{1}{1920}u_{xxx}^{2}+\frac{1}{80}uu_{x}u_{xxx}+\frac{1}{16}u^2u_{x}^{2}+\frac{1}{10}u^{5}+\frac{1}{24}u^{3}u_{xx}+\frac{1}{240}uu_{xx}^2 \nonumber \\
&\quad -\frac{1}{480}u_{x}^2u_{xx}-\frac{1}{4}u_{x}+\frac{3}{2}xu^{2}+\frac{1}{4}xu_{xx}-tu^{3}-\frac{1}{4}tuu_{xx}+\frac{1}{8}tu_{x}^{2}. \label{eq-represent-Hamiltonian-2}
\end{align}
Note that, the $\textrm{P}_{\textrm I}^2$ equation is the compatibility conditions of the following Hamiltonian system in two time variables (cf. \cite{Grava-Kapaev-Klein-2015,Suleimanov2014}),
\begin{align*}
\frac{d q_j}{dx} = \frac{\partial \mathcal{H}_1}{\partial p_j}, \qquad \frac{d p_j}{dx} = -\frac{\partial \mathcal{H}_1}{\partial q_j}, \quad j = 1,2, \\
\frac{d q_j}{dt} = \frac{\partial \mathcal{H}_2}{\partial p_j}, \qquad \frac{d p_j}{dt} = -\frac{\partial \mathcal{H}_2}{\partial q_j}, \quad j = 1,2, 
\end{align*}
with $\mathcal{H}_{1}=H_{1}, \mathcal{H}_{2}=-\frac{H_{2}}{3}$, $q_1 = u,$ $p_1 = \frac{1}{240} (u_{xxx} + 8 u u_x)$, $q_2 = \frac{1}{240}(u_{xx} + 6u^2)$ and $p_2 = u_x$.  We will build a full asymptotic expansion of $u(x,t)$, uniformly for $t|x|^{-\frac{2}{3}}\in(-\infty,M]$ with a constant $M\in(0,2^{-\frac{2}{3}}3^{-\frac{1}{3}})$. This expansion can be utilized to obtain the total integrals of $u(x,t)$ and $H_1(x,t)$ for $x$ on the real line. The investigation of the total integrations of  $u(x,t)$ and $H_1(x,t)$ is partially motivated by eigenvalue spacing problems in random matrix theory. It has been realized that, total integrals of several Painlev\'e transcendents or associated Hamiltonians are related to large gap asymptotics in certain random matrix models. Take the well-known Tracy-Widom distribution as a concrete example, which is explicitly expressed as 
\begin{equation} \label{eq:TW-PII}
  F_{\textrm{TW}}(x) = \exp \left( - \frac{1}{2} \int_x^{+\infty} (y-x) (q_{\HM}(y) )^2dy \right), \qquad x \in \mathbb{R},
\end{equation}
where $q_{\HM}(y)$ is the  Hastings-McLeod solution of the Painlev\'e II equation; see \cite{Tracy-Widom-1994}. The above formula can be rewritten as
\begin{equation} \label{eq:TW-PII-Ham}
  F_{\textrm{TW}}(x) = \exp \left( - \int_x^{+\infty} H(y) dy \right),
\end{equation}
where $
H(x)=-2^{1/3}H_{\textrm{PII}}(-2^{1/3}x)
$
with $H_{\textrm{PII}}$ being the Hamiltonian for the Painlev\'e II equation; see Forrester and  Witte \cite{FW01}, Tracy and Widom  \cite{Tracy-Widom-1994}. Therefore, the large gap asymptotics of  $ F_{\textrm{TW}}(x) $, including the constant term, can be achieved by deriving the total integrals of the Hastings-McLeod solution $q_{\HM}(y)$ or the associated Hamiltonian; see Baik et al. \cite{Baik-Its-2009}. This idea has been successfully applied to study large gap asymptotics for other matrix models; for example, see \cite{Bother:Buck2018,Bother:Its:Prok2019,Dai-Xu-Zhang-2020}.

The rest of this paper is organized as follows. In Section \ref{sec:main-results}, we state main results for the full asymptotic expansion of $u(x,t)$ in Theorem \ref{Thm-asym-u} and the total integrals of $u(x,t)$ and $H_{1}(x,t)$ in Theorem \ref{Thm-total-integral}. They are obtained via a steepest descent method for the associated Riemann-Hilbert (RH) problems. The detailed analysis is conducted in Section \ref{sec:RH-analysis}. Last, we prove  Theorem \ref{Thm-asym-u} and Theorem \ref{Thm-total-integral} in Section \ref{sec-proof-result}.

\section{Main results} \label{sec:main-results}

\subsection{Asymptotics of $u(x,t)$ and the associated Hamiltonians}
We first obtain a full asymptotic expansion for the tritronqu\'{e}e solution $u(x,t)$ of $\textrm{P}_{\textrm I}^2$ as $x\to\pm\infty$ below.

\begin{theorem}\label{Thm-asym-u}
Let $\mu=t|x|^{-\frac{2}{3}}$, and $z_{\pm}:=z_{\pm}(\mu)$ be the real root of 
\begin{equation} \label{zpm-def}
z_{\pm}^3-24\mu z_{\pm}\pm 48=0.
\end{equation}
Then, for any fixed $M\in(0,2^{-\frac{2}{3}}3^{-\frac{1}{3}})$, we have
\begin{equation}\label{eq-asym-u}
u(x,t)=\frac{z_{\pm}}{2}|x|^{\frac{1}{3}}+E(x,\mu) \quad \text{with}\quad E(x,\mu)\sim|x|^{\frac{1}{3}}\sum\limits_{k=1}^{\infty}\frac{e_{k}^{\pm}(\mu)}{|x|^{\frac{7k}{3}}},
\end{equation}
as $x\to\pm\infty$, where $e_{k}^{\pm}(\mu), k=1,2,\cdots$, are bounded uniformly for $\mu\in(-\infty,M]$. Moreover, for arbitrary  $k=1,2,\cdots$, we have 
\begin{equation} \label{eq:ek-large-mu}
e_{k}^{\pm}(\mu)= \mathcal{O}(\mu^{-2}) \qquad \textit{ as } \mu\to-\infty.
\end{equation}
\end{theorem}

The leading asymptotics of $u(x,t)$ as $x\to\pm\infty$ have been derived in \cite{Claeys-Vanlessen-2007, Grava-Kapaev-Klein-2015} for fixed $t$. When $t\to\pm\infty$ with $x=s|t|^{\frac{3}{2}}$, the leading asymptotics of $u(x,t)$ is provided in \cite{Claeys-2010} for $s$ in different intervals. A full asymptotic expansion of $u(x,t)$ is also given in \cite{Suleimanov} as follows:
\begin{equation}\label{eq-u-full-asym-fix-t}
u(x,t)=-6^{\frac{1}{3}}x^{\frac{1}{3}}-2\cdot 6^{-\frac{1}{3}} tx^{-\frac{1}{3}}+\sum\limits_{j=2}^{\infty}\frac{P_{j}(t)}{x^{\frac{j}{3}}} \qquad \textit{ as } x\to\pm\infty,
\end{equation}   
with $t$ being bounded. Note that, in the above formula and afterwards, $x^{\frac{1}{3}}=-|x|^{\frac{1}{3}}$ when $x<0$. Our expansion \eqref{eq-asym-u} is superior to the above one in the sense that it holds in a larger region of $t$ and is expressed as a power series of $|x|^{-\frac{7}{3}}$, rather than $x^{-\frac{1}{3}}$.

\begin{remark} \label{remark-1}
The coefficients $e_{k}^{\pm}(\mu)$ in \eqref{eq-asym-u} can be explicitly constructed term by term; see the derivation in Section \ref{sec-proof-result}. More precisely, we show that $e^{\pm}_k(\mu)$ are actually polynomials in terms of $(z_\pm^2 - 8 \mu)^{-1}$; for example,  
\begin{equation}
e_{1}^{\pm}(\mu)=-\frac{64}{3}(z_{\pm}^2-8\mu)^{-3}+\frac{256z_{\pm}^2}{3}(z_{\pm}^2-8\mu)^{-4}.
\end{equation} 
By choosing $M\in(0,2^{-\frac{2}{3}}3^{-\frac{1}{3}})$, one can ensure that $z_\pm^2 - 8 \mu$ is bounded away from 0 for all $\mu \in (-\infty, M]$. This  implies that $e_{k}^{\pm}(\mu), k=1,2,\cdots$, are analytic with respect to $\mu$ in the neighborhood of $(-\infty,M]$, and 
\begin{equation}\label{eq-error-bouond-E}
E(x,\mu)=\mathcal{O}((z_{\pm}^2-8\mu)^{-2}|x|^{-2}) \qquad \textit{ as } x\to\pm\infty.
\end{equation} 
Differentiating both side of \eqref{eq-asym-u} with respect to $x$ (keeping in mind that both $z_{\pm}$ and $\mu$ depend on $x$), one gets asymptotics for derivatives of $u(x,t)$ as follows:
\begin{eqnarray}
u_{x}(x,t)&=&\frac{-8}{z_{\pm}^2-8\mu}|x|^{-\frac{2}{3}}+\mathcal{O}((z_{\pm}^2-8\mu)^{-2}|x|^{-3}),  \label{eq-asym-ux-uniform} \\
u_{xx}(x,t)&=&\mathcal{O}((z_{\pm}^2-8\mu)^{-2}|x|^{-\frac{5}{3}}),\\
u_{xxx}(x,t)&=&\mathcal{O}((z_{\pm}^2-8\mu)^{-2}|x|^{-\frac{8}{3}}),
\end{eqnarray}
as $x\to\pm\infty$ uniformly for all $\mu\in(-\infty,M]$.
\end{remark}

Substituting \eqref{eq-asym-u} and \eqref{eq-asym-ux-uniform} into \eqref{eq-represent-Hamiltonian-1}, we obtain the asymptotics of the Hamiltonians.
\begin{corollary}\label{cor-asym-Hamiltonians}
Let $z_{\pm}$ be given in \eqref{zpm-def}, and the Hamiltonians $H_{1}(x,t), H_{2}(x,t)$ defined in \eqref{eq-represent-Hamiltonian-1} and \eqref{eq-represent-Hamiltonian-2}, respectively, then we have, as $x\to\pm\infty$,
\begin{align}
H_{1}(x,t)=&\left(\pm\frac{1}{4}z_{\pm}-\frac{z_{\pm}^4}{384}\right)|x|^{\frac{4}{3}}+\frac{4z_{\pm}}{3}\left(\frac{1}{z_{\pm}^2-8\mu}\right)^2|x|^{-1}+\mathcal{O}(|x|^{-\frac{10}{3}}), \label{eq-asym-Hamiltonians-uniformly-t-1} \\
H_{2}(x,t)=&\left(\frac{z_{\pm}^5}{320}\pm\frac{3z_{\pm}^2}{8}-\frac{\mu z_{\pm}^{3}}{8}\right)|x|^{\frac{5}{3}}+\frac{3z_{\pm}^2-8\mu}{(z_{\pm}^2-8\mu)^2}|x|^{-\frac{2}{3}}+\mathcal{O}(|x|^{-3}), \label{eq-asym-Hamiltonians-uniformly-t-2}
\end{align}
uniformly for all $\mu\in(-\infty,M]$. 
\end{corollary}
\begin{remark}\label{rem-error-H}
By a similar argument as in Remark \ref{remark-1}, the error terms in \eqref{eq-asym-Hamiltonians-uniformly-t-1} and \eqref{eq-asym-Hamiltonians-uniformly-t-2} are also analytic with respect to $\mu$ and  can be rewritten as $\mathcal{O}((z_{\pm}^2-8\mu)^{-2}|x|^{-\frac{10}{3}}), \mathcal{O}((z_{\pm}^2-8\mu)^{-2}|x|^{-3})$ respectively. As the error estimation is valid for $\mu \in (-\infty, M]$, {\it i.e.} $t \in (-\infty, M |x|^{\frac{2}{3}}]$, we can get a well-controlled error after integrating the expansion \eqref{eq-asym-Hamiltonians-uniformly-t-1}  with respect to $t$ from $- \infty$ to any positive fixed constant. This property will play an important role in our derivation for the total integral of $H_1(x,t)$ below.
\end{remark}

\begin{remark}
Since the above asymptotic results hold uniformly for all $\mu\in(-\infty,M]$, they are valid for any fixed $t$. Indeed, when $t$ is fixed, then $\mu\to 0$ as $x\to\pm\infty$. From the definition of $z_{\pm}$ in \eqref{zpm-def}, we find that 
\begin{equation}\label{eq-asymp-z-pm}
z_{+}(\mu)=-z_{-}(\mu)=-2\cdot 6^{\frac{1}{3}}-4\cdot 6^{-\frac{1}{3}}\mu+\frac{4\cdot 6^{\frac{1}{3}}}{27}\mu^{3}+\frac{4\cdot 6^{\frac{2}{3}}}{81}\mu^{4}+\mathcal{O}(\mu^{5})
\end{equation}
as $\mu\to 0$. Substituting this formula into \eqref{eq-asym-u} and \eqref{eq-asym-Hamiltonians-uniformly-t-1} respectively, we immediately obtain
\begin{equation}\label{eq-asym-u-fixed-t}
u(x,t)=-6^{\frac{1}{3}}x^{\frac{1}{3}}-2\cdot 6^{-\frac{1}{3}}tx^{-\frac{1}{3}}+\frac{2\cdot 6^{\frac{1}{3}}}{27}t^{3}x^{-\frac{5}{3}}+\frac{1}{36 x^2}+\mathcal{O}(x^{-\frac{7}{3}})
\end{equation}
and
\begin{equation}\label{eq-asym-Hamiltonians-fixed-t}
H_{1}(x,t)=-\frac{3}{4}6^{\frac{1}{3}}x^{\frac{4}{3}}-3\cdot 6^{-\frac{1}{3}}tx^{\frac{2}{3}}- t^2-\frac{6^{\frac{1}{3}}}{9}t^3 x^{-\frac{2}{3}}-\frac{1}{36x}+\mathcal{O}(x^{-\frac{4}{3}})
\end{equation}
as $x\to\pm\infty$ for any fixed $t$. 
\end{remark}

\subsection{Total integrals of $u(x,t)$ and $H_1(x,t)$}

It has been shown in \cite{Claeys-Vanlessen-2007} that $u(x,t)$ is pole-free on the real-$x$ axis for any $t \in \mathbb{R}$. From its definition in \eqref{eq-represent-Hamiltonian-1}, it is easy to see that $H_1(x,t)$ is pole-free for real $x$ as well. This makes it possible for us to compute their total integrals for $x \in \mathbb{R}$. There are two matters we need to consider. First, according to the asymptotic expansions in \eqref{eq-asym-u} and \eqref{eq-asym-Hamiltonians-uniformly-t-1}, both functions $u(x,t)$ and $H(x,t)$ do not decay sufficiently fast as $x \to \pm \infty$. Then, certain terms need to be deducted to make the integrals convergent. Second, as we don't have enough information about $u(x,t)$ and $H(x,t)$ when $x$ is finite, the asymptotics in \eqref{eq-asym-u} and \eqref{eq-asym-Hamiltonians-uniformly-t-1} only are not enough for us to establish the desired total integrals. Some differential identities are required to overcome this obstacle; for example, see some similar ideas in \cite{Baik-Its-2009,Dai-Xu-Zhang-2020, Kokocki-2020}.

For the problem we are addressing, the first differential identity we need is 
\begin{align}
\frac{\partial H_{1}(x,t)}{\partial x}&=u(x,t), \label{eq-differential-equality-1}
\end{align}
which is known in the literature. Next, keeping in mind that $u(x,t)$ also satisfies the KdV equation \eqref{eq-kdv}, a straightforward computation gives us an additional differential identity
\begin{align}
\frac{\partial H_{1}(x,t)}{\partial t}&=-\frac{1}{3}\frac{\partial H_{2}(x,t)}{\partial x}-\frac{1}{12}u_{xx}(x,t) = -\frac{1}{2} u(x,t)^2-\frac{1}{12}u_{xx}(x,t) ; \label{eq-differential-equality-2}
\end{align}
see the explicit expression of $H_1(x,t)$ and $H_2(x,t)$ in \eqref{eq-represent-Hamiltonian-1} and \eqref{eq-represent-Hamiltonian-2}. As the asymptotics of $H_{1}(x,t)$ is given in \eqref{eq-asym-Hamiltonians-fixed-t}, it is easy for us to obtain the total integral of $u(x,t)$ with the aid of \eqref{eq-differential-equality-1}. Derivation of the total integral of $H_{1}(x,t)$ is more involving. Let us denote
\begin{equation} \label{eq-def-I(t)}
I(t)=\int_{X_{1}}^{X_{2}}H_{1}(x,t)dx
\end{equation}
with two arbitrary  constants $X_{k}$, $k=1,2$. Then, with the second differential identity \eqref{eq-differential-equality-2}, we have
\[ \frac{d I(t)}{dt}  = \int_{X_{1}}^{X_{2}} \frac{\partial H_{1}(x,t)}{\partial t} dx =  \int_{X_{1}}^{X_{2}} \left[ -\frac{1}{3}\frac{\partial H_{2}(x,t)}{\partial x}-\frac{1}{12}u_{xx}(x,t) \right] dx.\]
Integrating both sides of the above formula from $-\infty$ to $t$, we get 
\begin{equation}\label{eq-integral-with-t}
I(t)=I(-\infty)+\int_{-\infty}^{t}\left(\frac{1}{3}H_{2}(X_{1}, \tau)+\frac{1}{12}u_{x}(X_{1},\tau)-\frac{1}{3}H_{2}(X_{2},\tau)-\frac{1}{12}u_{x}(X_{2},\tau)\right)d\tau.
\end{equation}
In such a way, we transform an integral of $H_1(x,t)$ with respect to $x$ into a new one with respect to $t$. With delicate uniform asymptotics for both $u(x,t)$ and $H_2(x,t)$ as $x \to \pm \infty$, we successfully establish the total integral of $H_{1}(x,t)$. 

\begin{remark}
To get a total integral of $H_{1}(x,t)$, we will let $X_1 \to -\infty$ and $X_2 \to +\infty$.  By \eqref{eq-integral-with-t}, this requires the asymptotics for both $u_{x}(x,\tau)$ and $H_{2}(x,\tau)$ as $x\to\pm\infty$ hold uniformly for all $\tau\in(-\infty,t]$. It is one of the motivations why we first establish the corresponding asymptotic expansion of $u(x,t)$ in Theorem \ref{Thm-asym-u}.
\end{remark}

We are now in position to state our second main results about the total integrals of $u(x,t)$ and the associated Hamiltonian $H_{1}(x,t)$. 

\begin{theorem}\label{Thm-total-integral}
Let $u(x,t)$ be the tritronqu\'{e}e solution of the $\textrm{P}_{\textrm I}^2$ equation \eqref{eq-PI-hierarchy} and $H_{1}(x,t)$ be the associated Hamiltonian given in \eqref{eq-represent-Hamiltonian-1}. Then, for any fixed $t\in\mathbb{R}$, we have
\begin{equation}\label{eq-total-integral-u}
\int_{-\infty}^{+\infty}\left(u(x,t)+6^{\frac{1}{3}}x^{\frac{1}{3}}+2\cdot 6^{-\frac{1}{3}} t x^{-\frac{1}{3}}\right)dx=0
\end{equation}
and
\begin{equation}\label{eq-total-integral-H1}
\int_{-\infty}^{+\infty}\left(H_{1}(x,t)+\frac{3}{4}6^{\frac{1}{3}}x^{\frac{4}{3}}+3\cdot 6^{-\frac{1}{3}}tx^{\frac{2}{3}}+t^2+\frac{6^{\frac{1}{3}}t^3}{9}x^{-\frac{2}{3}}+\frac{x}{36(x^2+1)}\right)=0.
\end{equation}
\end{theorem}

\begin{remark}
Due to the asymptotics of $u(x,t)$ and $H_1(x,t)$ given in \eqref{eq-asym-u-fixed-t} and \eqref{eq-asym-Hamiltonians-fixed-t}, some leading terms are deducted to ensure the convergence of the integral in the above theorem. Moreover, in \eqref{eq-total-integral-H1}, we adopt the term $\frac{x}{36(x^2+1)}$, instead of $\frac{1}{36x}$, to ensure that the integral is convergent at $x=0$. It should also be noted that both $u(x,t)$ and $H_{1}(x,t)$ are functions of $(x,t)$, which implies that their total integrals about $x$ should dependent on $t$ in general. Surprisingly, the right-hand side of \eqref{eq-total-integral-u} and \eqref{eq-total-integral-H1} are both zero, despite the fact that $u(x,t)$ behaves dramatically differently for $t>0$ and $t<0$; see \cite{Claeys-2010} or \cite[Fig. 19]{Grava-Kapaev-Klein-2015}. 
\end{remark}

By differentiating both sides of \eqref{eq-total-integral-H1} with respect to $t$ and taking into account the differential identity in \eqref{eq-differential-equality-2}, we also derive the total integral of $u(x,t)^2$.
\begin{corollary}
For any fixed $t\in\mathbb{R}$, we have 
\begin{equation}
\int_{-\infty}^{+\infty}\left(u(x,t)^2-6^{\frac{2}{3}}x^{\frac{2}{3}}-4t-4\cdot 6^{-\frac{2}{3}}t^2 x^{-\frac{2}{3}}\right)dx=0.
\end{equation}
\end{corollary}

Since $\frac{\partial H_{1}(x,t)}{\partial x}=u(x,t)$, an integration by par gives us
\begin{equation}
\int_{X_{1}}^{X_{2}}x \cdot u(x,t)dx=X_{2}H_{1}(X_{2},t)-X_{1}H_{1}(X_{1},t)-\int_{X_{1}}^{X_{2}}H_{1}(x,t)dx.
\end{equation} 
Taking the limits $X_{1}\to-\infty$, $X_{2}\to+\infty$, and making use of the asymptotic of $u(x,t)$ in \eqref{eq-asym-u-fixed-t} and the total integral of $H_{1}(x,t)$ in \eqref{eq-total-integral-H1}, we obtain a total integral for $xu(x,t)$.
\begin{corollary}
For any fixed $t\in\mathbb{R}$, we have
\begin{equation}
\int_{-\infty}^{+\infty}\left(x\cdot u(x,t)+6^{\frac{1}{3}}x^{\frac{4}{3}}+2\cdot 6^{-\frac{1}{3}}t x^{\frac{2}{3}}-\frac{x}{36(x^2+1)}\right)dx=0.
\end{equation}
\end{corollary}

\section{RH analysis for the Painlev\'{e} I hierarchy} \label{sec:RH-analysis}
In this section, we perform RH analysis to obtain the uniform asymptotic expansion stated in Theorem \ref{Thm-asym-u}, using an approach similar to \cite{Claeys-2010, Claeys-Vanlessen-2007}. Our goal is to derive a complete asymptotic expansion and demonstrate that the coefficients $e_{k}^{\pm}(\mu)$ are analytic in a neighborhood of $(-\infty,M]$ with respect to $\mu$. Therefore, we will pay more attention to the influence of $\mu$ in the RH analysis. To keep it concise, we will only present the essential steps here and refer to \cite{Claeys-2010, Claeys-Vanlessen-2007} for more detailed information. Additionally, we will only analyze the case of $x\to+\infty$, as it is similar to the case of $x\to-\infty$.

\subsection{The steepest descent analysis}

The tritronqu\'{e}e solution $u(x,t)$ of $\textrm{P}_{\textrm I}^2$ equation is associated with a RH problem as follows; see \cite{FAS-2006} for more detailed information about the Painlev\'e equations and RH problems.
\\

\noindent\textbf{RH problem for $Y(\lambda)$}

\begin{enumerate}
\item[(Y1)] $Y(\lambda)$ is analytic for $\lambda \in \mathbb{C}\setminus\Gamma$, where $\Gamma=\cup_{k=1}^{4}\Gamma_{4}$ is described in Figure \ref{Y-contour};

\begin{figure}[h]
\centering
\includegraphics[scale=1.5]{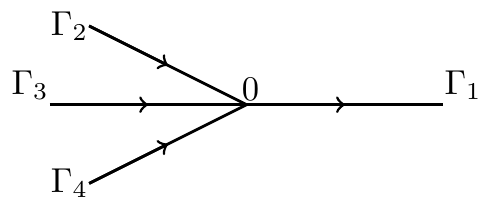} 
\caption{Contours for the RH problem for $Y$} \label{Y-contour}
\end{figure}

\item[(Y2)] $Y(\lambda)$ satisfies the following jump conditions on $\Gamma$:
\begin{equation}
Y_{+}(\lambda)=Y_{-}(\lambda)\begin{cases}\left(\begin{matrix}
1 & 1\\ 0& 1
\end{matrix}\right), \qquad \qquad &\lambda\in\Gamma_{1},\\
\left(\begin{matrix}
1 & 0\\ 1& 1
\end{matrix}\right), \qquad \qquad &\lambda\in\Gamma_{2}\cup\Gamma_{4},\\
\left(\begin{matrix}
0 & 1\\ -1& 0
\end{matrix}\right), \qquad \qquad &\lambda\in\Gamma_{3};
\end{cases}
\end{equation}

\item[(Y3)] When $\lambda\to\infty$, the asymptotic expansion of $Y(\lambda)$ is
\begin{equation}\label{eq-asym-Y}
Y(\lambda)\sim\left(I+\sum\limits_{k=1}^{\infty}A_{k}(x,t)\lambda^{-k}\right)\lambda^{-\frac{1}{4}\sigma_{3}}Ne^{-\theta(\lambda;x,t)\sigma_{3}},
\end{equation}
where 
\begin{equation}
N=\frac{1}{\sqrt{2}}\left(\begin{matrix}
1&1\\-1&1
\end{matrix}\right)e^{-\frac{1}{4}\pi i\sigma_{3}},\qquad \qquad \sigma_{3}=\left(\begin{matrix}
1&0\\0&-1
\end{matrix}\right)
\end{equation}
and
\begin{equation} \label{eq-theta-def}
\theta(\lambda;x,t)=\frac{1}{105}\lambda^{\frac{7}{2}}-\frac{1}{3}t\lambda^{\frac{3}{2}}+x\lambda^{\frac{1}{2}}.
\end{equation}
In the above formulas, the branch is chosen such that $\arg{\lambda}\in(-\pi,\pi)$. 
\end{enumerate}

It has been proved in \cite{Claeys-Vanlessen-2007} that the RH problem for $Y(\lambda)$ is uniquely solvable for all real values of $x$ and $t$. The tritronqu\'{e}e solution $u(x,t)$ for the $\textrm{P}_{\textrm I}^2$ equation \eqref{eq-PI-hierarchy} is given explicitly by 
\begin{equation}\label{eq-u-by-A1}
u(x,t)=2(A_{1})_{11}-(A_{1})_{12}^{2},
\end{equation}
where $A_1= A_1(x,t)$ is the coefficient matrix in \eqref{eq-asym-Y}, and $(A_{1})_{ij}$ denotes its $(i,j)$-entry.

To investigate the asymptotics of the RH problem for $Y(\lambda)$ as $x \to +\infty,$ let us first make the rescaling $t=\mu |x|^{\frac{2}{3}}$ and $ \lambda=\zeta |x|^{\frac{1}{3}}$. Then, the function $\theta(\cdot)$ in \eqref{eq-theta-def} can be written as
\begin{equation}
\theta(|x|^{\frac{1}{3}}\zeta; x,t)=|x|^{\frac{7}{6}}\left(\frac{1}{105}\zeta^{\frac{7}{2}}-\frac{\mu}{3}\zeta^{\frac{3}{2}}+\zeta^{\frac{1}{2}}\right).
\end{equation}
Next, in order to normalize the large-$\lambda$ behavior in \eqref{eq-asym-Y}, we introduce the following $g$-function:
\begin{equation}\label{eq-def-g}
g(\zeta)=c_{1}(\zeta-z_{+})^{\frac{7}{2}}+c_{2}(\zeta-z_{+})^{\frac{5}{2}}+c_{3}(\zeta-z_{+})^{\frac{1}{2}}, \qquad \zeta \in \mathbb{C} \setminus (-\infty, z_+]
\end{equation}
with $\arg(\zeta-z_{+})\in(-\pi,\pi)$, where $z_{+}=z_{+}(\mu)$ is the real root of $z^3-24\mu z+48=0$. The constants $c_k$ in the above formula are chosen to be 
\begin{equation}\label{eq-c1-c2-c3}
c_{1}=\frac{1}{105},\qquad c_{2}=\frac{z_{+}}{30},\qquad c_{3}=-\frac{\mu}{3}+\frac{1}{24}z_{+}^2,
\end{equation} 
such that 
\begin{equation}\label{eq-compare-g-theta}
|x|^{\frac{7}{6}}g(\zeta)=\theta(|x|^{\frac{1}{3}}\zeta; x,t)+d_{1}\zeta^{-\frac{1}{2}}+\mathcal{O}(\zeta^{-\frac{3}{2}}), \qquad \text{as } \zeta\to\infty,
\end{equation}
with $d_{1}$ being a constant independent of $\zeta$.  

To facilitate our subsequent analysis, let us list some important properties of $g(\zeta)$ below.

\begin{lemma}\label{lem-g}
Given any constant $M$ with $0<M<M_{+}=3^{\frac{5}{3}}2^{-\frac{2}{3}}5^{-\frac{1}{3}}$, we have, for all $\mu\in(-\infty,M]$, 
\begin{enumerate}
\item [(i)] $g(\zeta)>0$ when $\zeta>z_{+}$;
\item [(ii)] $\re g_{\pm}(\zeta)\equiv 0$ and $\im g_{\pm}(\zeta)\neq 0$ when $\zeta<z_{+}$;
\item [(iii)] $\im g_{+}'(\zeta)>0$ when $\zeta<z_{+}$;
\item [(iv)] $\frac{1}{|g(\zeta)|}=\mathcal{O}(z_{+}^2-8\mu)^{-1}$ as $\mu\to-\infty$ uniformly for all $\zeta$ bounded away from $z_{+}$.
\end{enumerate}
\end{lemma} 
\begin{proof}
It is easily seen from its definition in \eqref{eq-def-g} that $g(\zeta)$ is real when $\zeta>z_{+}$ and $\re g_{\pm}(\zeta) = 0$ when $\zeta<z_{+}$. Next, we establish (i) and (ii) by proving  $z_{+}$ is the unique real zero of $g(\zeta)$. This can be justified by imposing the condition $c_{2}^2-4c_{1}c_{3}<0$, which implies $\mu<\frac{3}{80}z_{+}^2$. Recalling $\mu=\frac{z_{+}^3+48}{24z_{+}}$, then we have $-480^{\frac{1}{3}}<z_{+}<0$. This agrees with the condition $\mu\in(-\infty, M]$. In such a way, we prove (i) and (ii).

To prove (iii), we first have from \eqref{eq-def-g} that 
\begin{equation}
g'(\zeta)= (\zeta-z_{+})^{\frac{1}{2}}\left[\frac{1}{30}(\zeta-z_{+})^{2}+\frac{1}{12}z_{+}(\zeta-z_{+})-\left(\frac{\mu}{2}-\frac{z_{+}^2}{16}\right)\right].
\end{equation}
Based on the derivation in the previous paragraph, it is clear that $z_+< 0$ and $\mu<\frac{3}{80}z_{+}^2$. Then, we get
\begin{equation}
\im g_{+}'(\zeta)\geq|z_{+}-\zeta|^{\frac{1}{2}}\left(\frac{z_{+}^2}{16}-\frac{\mu}{2}\right)>0, \qquad \textrm{when }\zeta<z_{+}.
\end{equation}
Last, it follows from \eqref{eq-def-g} that $\frac{1}{|g(\zeta)|} = \mathcal{O}(c_3^{-1}) $ as $\mu \to \infty$. Recalling the definition of $c_3$ in \eqref{eq-c1-c2-c3}, we immediately get (iv).
\end{proof}

\begin{remark}\label{remark-g}
From (iii) in the above lemma, a straightforward derivation using the Cauchy-Riemann equations  implies that there exists a constant $\theta_{0}$, such that $\re g(\zeta)<0$ when $\arg(\zeta-z_{+})=\pm\pi+\theta_{0}$.
\end{remark}

\begin{remark}
When considering the asymptotics as $x\to-\infty$, the role of $z_{+}$ will be replaced by $z_{-}$, the real root of $z^3-24\mu z-48=0$. In this case, the upper bound of $M$ in the above lemma should be replaced by  $M_{-}=2^{-\frac{2}{3}}3^{-\frac{1}{3}}$. Therefore, to ensure the asymptotic expansions in Theorem \ref{Thm-asym-u} are valid as $x\to\pm\infty$, we set $M\in(0,2^{-\frac{2}{3}}3^{-\frac{1}{3}})$. 
\end{remark}

\subsection*{Normalization: $Y \mapsto T$}

Define
\begin{equation}
T(\zeta)=|x|^{\frac{1}{12}\sigma_{3}}\left(\begin{matrix}
1&0\\d_{1}|x|^{\frac{1}{6}}&1
\end{matrix}\right)
\begin{cases}
Y(|x|^{\frac{1}{3}}\zeta)e^{|x|^{\frac{7}{6}}g(\zeta)\sigma_{3}}, & \zeta\in I\cup II\cup III,\\
Y(|x|^{\frac{1}{3}}\zeta)\left(\begin{matrix}1&0\\1&1\end{matrix}\right)e^{|x|^{\frac{7}{6}}g(\zeta)\sigma_{3}}, & \zeta\in I'\cup I''.
\end{cases}
\end{equation}
where $d_{1}$ is given in \eqref{eq-compare-g-theta} and the regions $I, I', I'', II, III$ are described in Fig. \ref{fig-T}. A direct calculation from the RH problem for $Y(\lambda)$ leads to the following RH problem for $T(\zeta)$.

\smallskip

\begin{figure}[h]
\centering
\includegraphics[scale=1.25]{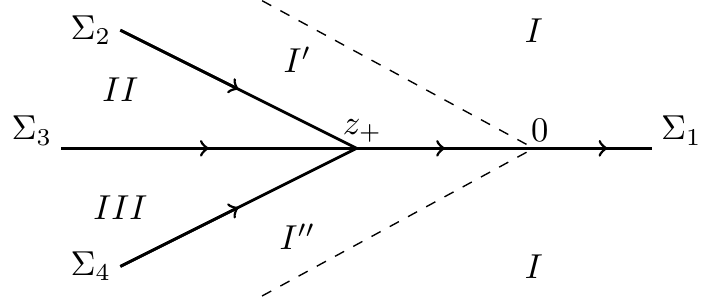} 
\caption{Regions and contours for the RH problem for $T$} \label{fig-T}
\end{figure}

\smallskip 

\noindent\textbf{RH problem for $T(\zeta)$}

\begin{enumerate}
\item[(T1)] $T(\zeta)$ is analytic in $\mathbb{C}\setminus\Sigma$, where $\Sigma=\cup_{k=1}^{4}\Sigma_{k}$ is described in Fig. \ref{fig-T};

\item[(T2)] $T(\zeta)$ satisfies the following jump conditions on $\Sigma$:
\begin{equation} \label{eq: t-jump}
T_{+}(\zeta)=T_{-}(\zeta)\begin{cases}
\left(\begin{matrix}
1 & e^{-2|x|^{\frac{7}{6}}g(\zeta)}\\ 0 & 1
\end{matrix}\right), &\qquad \zeta\in\Sigma_{1},\\
\left(\begin{matrix}
1 & 0\\ e^{2|x|^{\frac{7}{6}}g(\zeta)} & 1
\end{matrix}\right), &\qquad \zeta\in\Sigma_{2}\cup\Sigma_{4},\\
\left(\begin{matrix}
0 & 1\\ -1 & 0
\end{matrix}\right), &\qquad \zeta\in\Sigma_{3}.
\end{cases}
\end{equation}

\item[(T3)] As $\zeta \to \infty$,  we have
\begin{equation} \label{eq:T-large}
T(\zeta)=\left(I+\frac{T_{1}(x,t)}{\zeta}+\mathcal{O}(\zeta^{-2})\right)\zeta^{-\frac{1}{4}\sigma_{3}}N,
\end{equation}
with 
\begin{equation}\label{eq-def-B1}
T_{1}(x,t)=\left(\begin{matrix}
\frac{d_{1}^2}{2}+|x|^{-\frac{1}{3}}(A_{1})_{11}-d_{1}|x|^{-\frac{1}{6}}(A_{1})_{12} &\qquad |x|^{-\frac{1}{6}}(A_{1})_{12}-d_{1}\\
*  &  *
\end{matrix}\right),
\end{equation}
where $A_{1}$ is given in \eqref{eq-asym-Y} and $*$ denotes some unimportant entries. 
\end{enumerate}

A combination of \eqref{eq-u-by-A1} and \eqref{eq-def-B1} yields that 
\begin{equation}\label{eq-u-by-T}
u(x,t)=2|x|^{\frac{1}{3}}(T_{1})_{11}-|x|^{\frac{1}{3}}(T_{1})_{12}^2.
\end{equation}

\subsection*{Global parametrix}
From Lemma \ref{lem-g} and Remark \ref{remark-g}, one can see the jump matrices in \eqref{eq: t-jump} tend to the identity matrix when $x \to +\infty$, except for the one on $\Sigma_3$. This gives us a global parametrix for $P^{(\infty)}(\zeta)$, only possessing a jump on $(-\infty,z_{+})$.

\medskip 

\noindent\textbf{RH problem for $P^{(\infty)}(\zeta)$}
\begin{itemize}
\item $P^{(\infty)}(\zeta)$ is analytic for $\zeta \in \mathbb{C}\setminus (-\infty,z_{+}]$;

\item For $\zeta \in (-\infty,z_{+})$, we have $
P_{+}^{(\infty)}(\zeta)=P_{-}^{(\infty)}(\zeta)\left(\begin{matrix}
0&1\\-1&0
\end{matrix}\right)$;

\item As $\zeta \to \infty$, we have $
P^{(\infty)}(\zeta)=\left(I+\mathcal{O}(\zeta^{-1})\right)\zeta^{-\frac{1}{4}\sigma_{3}}N.$
\end{itemize}

It is easy to check that  the above RH problem possesses a solution as follows
\begin{equation} \label{eq-Pinf-def}
P^{(\infty)}(\zeta)=(\zeta-z_{+})^{-\frac{\sigma_{3}}{4}}N, \qquad \zeta \in \mathbb{C}\setminus(-\infty,z_{+}].
\end{equation}
Moreover, via a straightforward computation, we have from \eqref{eq:T-large} and the above formula
\begin{equation}
T(\zeta)P^{(\infty)}(\zeta)^{-1}=I+\left(T_{1}-\frac{z_{+}\sigma_{3}}{4}\right)\frac{1}{\zeta}+  \mathcal{O} (\zeta^{-2}) \qquad \textrm{as } \zeta\to\infty.
\end{equation}

\subsection*{Local parametrix near $z_+$}


Denote $U(z_{+},\delta)$ by the neighborhood  $\{\zeta: |\zeta-z_{+}|\leq\delta\}$, where $\delta$ is a small and fixed constant. In $U(z_{+},\delta)$, a local parametrix $P(\zeta)$ is constructed to approximate $T(\zeta)$ in terms of the Airy functions; for example, see \cite[Section 2.4]{Claeys-2010}. More precisely, we have 
\begin{equation} \label{eq-local-para}
P(\zeta)=(\zeta-z_{+})^{-\frac{\sigma_{3}}{4}}\left(|x|^{\frac{7}{9}}f(\zeta)\right)^{\frac{\sigma_{3}}{4}}M(|x|^{\frac{7}{9}}f(\zeta)),
\end{equation}
where 
\begin{equation}\label{eq-z-def}
f(\zeta)=\left(\frac{3}{2}g(\zeta)\right)^{\frac{2}{3}}
\end{equation} 
is a conformal mapping in $U(z_{+},\delta)$ and $M(\cdot)$ is the standard Airy parametrix; see the explicit RH problem for $M$ in \cite[Section 2.4]{Claeys-Vanlessen-2007}. Here, we omit its details and only record its large-$z$ behaviour below:
\begin{equation} \label{eq-M-asymp}
M(z)\sim \left(I+\sum\limits_{k=1}^{\infty}B_{k}z^{-k}\right)z^{-\frac{1}{4}\sigma_{3}}N, \qquad \textrm{as } z\to\infty,
\end{equation}
where $N$ is given in \eqref{eq-asym-Y} and 
\begin{equation}
B_{3k-2}=\left(\begin{matrix}
0&0\\t_{2k-1}&0
\end{matrix}\right),\qquad B_{3k-1}=\left(\begin{matrix}
0&\hat{t}_{2k-1}\\0&0
\end{matrix}\right),\qquad B_{3k}=\left(\begin{matrix}
\hat{t}_{2k}&0\\0&t_{2k}
\end{matrix}\right)
\end{equation}
with 
\begin{equation}\label{eq-tk-hat-tk}
\hat{t}_{k}=\frac{\Gamma(3k+1/2)}{36^{k}k!\Gamma(k+1/2)
},\qquad \qquad t_{k}=-\frac{6k+1}{6k-1}\hat{t}_{k}.
\end{equation}

%

%

\subsection*{Final transformation $T \mapsto R$}

The final transformation is defined as 
\begin{equation}
R(\zeta)=\begin{cases}T(\zeta)P^{(\infty)}(\zeta)^{-1},\qquad& \text{for }\zeta\in(\mathbb{C}\setminus U(z_{+},\delta))\setminus \Sigma,\\
T(\zeta)P(\zeta)^{-1},\qquad &\text{for } \zeta\in U(z_{+},\delta)\setminus\Sigma.
\end{cases}
\end{equation}
Then $R(\zeta)$ satisfies the following RH problem.

\medskip

\noindent\textbf{RH problem for $R(\zeta)$}

\begin{enumerate}
\item[(R1)] $R(\zeta)$ is analytic in $\mathbb{C}\setminus\Sigma_{R}$, where $\Sigma_{R}=\Sigma_{R,1}\cup\Sigma_{R,2}\cup\Sigma_{R,4}\cup\partial U(z_{+},\delta)$ is described in Figure \ref{fig-R};
\begin{figure}[h]
\centering\includegraphics[scale=1]{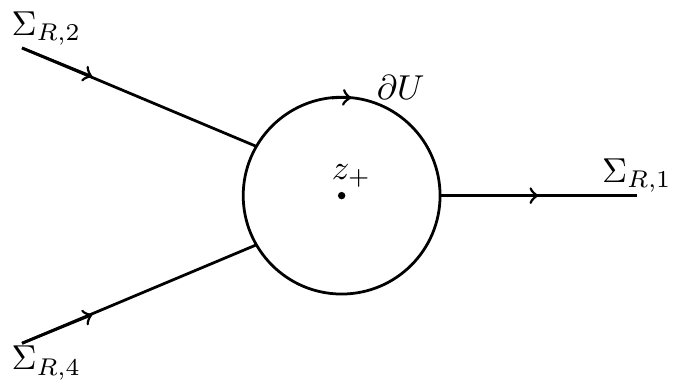}
\caption{Contours for the RH problem for  $R$.} \label{fig-R}
\end{figure}

\item[(R2)] When $\zeta\in\Sigma_{R}$, we have 
\begin{equation}\label{eq-jump-R}
R_{+}(\zeta)=R_{-}(\zeta)J_R(\zeta)
\end{equation}
where 
\begin{equation} \label{eq-JR-formula}
J_R(\zeta)=\begin{cases}P^{(\infty)}(\zeta)\left(\begin{matrix}
1 & e^{-2|x|^{\frac{7}{6}}g(\zeta)}\\ 0 & 1
\end{matrix}\right)P^{(\infty)}(\zeta)^{-1}, \qquad & \zeta\in\Sigma_{R,1},\\
P^{(\infty)}(\zeta)\left(\begin{matrix}
1 &0 \\ e^{2|x|^{\frac{7}{6}}g(\zeta)} & 1
\end{matrix}\right)P^{(\infty)}(\zeta)^{-1}, \qquad & \zeta\in\Sigma_{R,2}\cup\Sigma_{R,4}, \\
P(\zeta)P^{(\infty)}(\zeta)^{-1}, &  \zeta\in\partial{U(z_{+},\delta)};
\end{cases}
\end{equation}

\item[(R3)]As $\zeta\to\infty$, we have 
\begin{equation}\label{eq-R3-C1-T1}
R(\zeta)=I+\frac{C_{1}(x,t)}{\zeta}+\mathcal{O}(\zeta^{-2}), \quad \text{where } C_{1}=T_{1}-\frac{z_{+}\sigma_{3}}{4}.
\end{equation}
\end{enumerate}

Next, we will conduct a detailed asymptotic analysis on $R(\zeta)$ to establish a full asymptotic expansion \eqref{eq-asym-u} in Theorem \ref{Thm-asym-u}, as well as the analyticity of the coefficients $e_{k}^{\pm}(\mu)$. This part is new comparing with the existing RH analysis related to the tritronqu\'{e}e solution $u(x,t)$.

\subsection{Asymptotics of $R(\zeta)$}

As $x \to +\infty$, the jump matrix $J_R(\zeta)$ in \eqref{eq-jump-R} decays exponentially to the identical matrix uniformly on $\Sigma_{R,1}\cup\Sigma_{R,2}\cup\Sigma_{R,4}$. On $\partial U(z_{+},\delta)$, it follows from \eqref{eq-Pinf-def}, \eqref{eq-local-para}, \eqref{eq-M-asymp} and \eqref{eq-JR-formula} that
\begin{equation}\label{eq-v2-def}
\begin{split}
J_R(\zeta) = P(\zeta)P^{(\infty)}(\zeta)^{-1} &\sim I+\sum\limits_{k=1}^{\infty}Q_{k}(\zeta)|x|^{-\frac{7}{6}k}, \qquad \textrm{as } x\to+\infty,
\end{split}
\end{equation}
with 
\begin{equation}\label{eq-def-Qk}
\begin{split}
Q_{2k}(\zeta)=&B_{3k}\left(\frac{3}{2}g(\zeta)\right)^{-2k}=\left(\begin{matrix}
\hat{t}_{2k}&0\\0&t_{2k}
\end{matrix}
\right)\left(\frac{3}{2}g(\zeta)\right)^{-2k}, \\
Q_{2k-1}(\zeta)=&\left(\frac{B_{3k-1}}{(\zeta-z_{+})^{\frac{1}{2}}}+B_{3k-2}(\zeta-z_{+})^{\frac{1}{2}}\right)\left(\frac{3}{2}g(\zeta)\right)^{-2k+1}\\
=&\left(\begin{matrix}
0&\frac{\hat{t}_{2k-1}}{(\zeta-z_{+})^{\frac{1}{2}}}\\t_{2k-1}(\zeta-z_{+})^{\frac{1}{2}}&0
\end{matrix}
\right)\left(\frac{3}{2}g(\zeta)\right)^{-2k+1},
\end{split}
\end{equation} 
where $t_{k}$ and $\hat{t}_{k}$ are constants given in \eqref{eq-tk-hat-tk}. This implies that the RH problem for $R$ is a small norm problem when $x \to +\infty$. By a standard argument for RH problems \cite{Deift-book},   $R(\zeta)$ exists when $x$ is large enough and admits an asymptotic expansion as follows:
\begin{equation}\label{eq-R-series}
R(\zeta)\sim I+\sum\limits_{k=1}^{\infty}\frac{R_{k}(\zeta)}{x^{\frac{7}{6}k}}, \quad\text{ as }x\to+\infty,
\end{equation}
uniformly for $\zeta \in \mathbb{C} \setminus \Sigma_R$.

It is possible to get more information about $R_{k}(\zeta)$ in the above expansion. Substituting \eqref{eq-v2-def} and \eqref{eq-R-series} into \eqref{eq-jump-R}, we have, for $k \in \mathbb{N}^{+} $,
\begin{equation}\label{eq-Rk-reccursion}
R_{k,+}(\zeta)=R_{k,-}(\zeta)+\sum\limits_{m=0}^{k-1}R_{m,-}(\zeta)Q_{k-m}(\zeta), \qquad \zeta\in\partial{U(z_{+},\delta)},
\end{equation} 
with $R_{0}(\zeta) \equiv I$. First, we have the following result.

\begin{lemma}\label{lem-Rk-properties}
For any $k\in\mathbb{N}^{+}$, $R_{2k}(\zeta)$ is a diagonal matrix and $R_{2k-1}(\zeta)$ is an anti-diagonal matrix. 
\end{lemma}

\begin{proof}
Note that the coefficient matrices $Q_k(\zeta)$ in \eqref{eq-v2-def} possess diagonal and anti-diagonal structures; see \eqref{eq-def-Qk}. Using \eqref{eq-Rk-reccursion}, we obtain the lemma by mathematical induction.
\end{proof}

The next lemma give us the analyticity of $R_{k}(\zeta)$ with respect to $\mu$ in a neighborhood of $(-\infty,M]$.

\begin{lemma}\label{lem-Rk-analytic}
Denote $n_{2k}=3k,~  n_{2k-1}=3k-1$ for all $k\in\mathbb{N}^{+} $. We have the following representation for the coefficients $R_{k}(\zeta)$ in \eqref{eq-R-series}:
\begin{equation}\label{eq-Rk-Laurent-representation}
R_{k}(\zeta)=\begin{cases}\sum\limits_{m=1}^{n_{k}}\frac{R_{k}^{(-m)}}{(\zeta-z_{+})^{m}},& \zeta\in\mathbb{C}\setminus{\overline{U(z_{+},\delta)}},\smallskip \\
\sum\limits_{m=0}^{+\infty}R_{k}^{(m)}(\zeta-z_{+})^{m},&\zeta\in U(z_{+},\delta),
\end{cases}
\end{equation}
where $R_{k}^{(m)}$ are polynomials of $\frac{1}{c_{3}}$  and analytic with respect to $\mu$ in a neighborhood of $(-\infty,M]$ for all $k \in \mathbb{N}^+$ and $m=-n_{k},-n_{k}+1,\cdots$. In particular, we have
\begin{eqnarray}
R_{1}^{(-1)}&=&\left(\begin{matrix}
0& -\frac{2c_{2}\hat{t}_{1}}{3c_{3}^2}\\ \frac{2t_{1}}{3c_{3}}&0
\end{matrix}\right), \label{eq:R1--1} \\
R_{2}^{(-1)}&=&\left(\begin{matrix}
\frac{4\hat{t}_{2}(3c_{2}^{2}-2c_{1}c_{3})-4t_{1}\hat{t}_{1}(c_{2}^2-c_{1}c_{3})}{9c_{3}^{4}}&0\\
0&\frac{4t_{2}(3c_{2}^{2}-2c_{1}c_{3})-4t_{1}\hat{t}_{1}(2c_{2}^2-c_{1}c_{3})}{9c_{3}^{4}}
\end{matrix}\right), \label{eq:R2--1} 
\end{eqnarray}
where $c_{k}$, $t_{k}$ and $\hat{t}_{k}$ are given  in \eqref{eq-c1-c2-c3} and \eqref{eq-tk-hat-tk}, respectively.
\end{lemma}

\begin{proof}
When $k=1$,  the equation \eqref{eq-Rk-reccursion} is simply $R_{1,+}(\zeta)=R_{1,-}(\zeta)+Q_{1}(\zeta)$ for $\zeta\in\partial{U(z_{+},\delta)}$. It then follows from the Plemelj formula that 
\begin{equation}\label{eq-R1-by-integral}
R_{1}(\zeta)=\frac{1}{2\pi i}\oint_{\partial U(z_+, \delta)}\frac{Q_{1}(s)}{s-\zeta}ds,
\end{equation}
where the orientation of $\partial U(z_+, \delta)$ is depicted in Figure \ref{fig-R}. From  \eqref{eq-def-g} and \eqref{eq-def-Qk}, one can see that $Q_{1}(\zeta)$ has a pole of order 2 at $\zeta=z_+$. Denote its Laurent series at $\zeta=z_+$ by 
\begin{equation}\label{eq-Q1-Laurent-series}
Q_{1}(\zeta)=\sum\limits_{m=-2}^{+\infty}\frac{Q_{1}^{(m)}}{(\zeta-z_{+})^m}, \qquad |z-z_+| > 0, 
\end{equation}
where $Q_{1}^{(m)}, m=-2,-1,0,\cdots$, are polynomials of $\frac{1}{c_{3}}$. Combining the above two formulas, we have 
\begin{equation}
R_{1}(\zeta)=\begin{cases}
\frac{R_{1}^{(-2)}}{(\zeta-z_{+})^2}+\frac{R_{1}^{(-1)}}{\zeta-z_{+}},\qquad & \zeta\in\mathbb{C}\setminus\overline{U(z_+, \delta)}, \smallskip \\
\sum\limits_{m=0}^{+\infty}R_{1}^{(m)}(\zeta-z_{+})^{m},\qquad &\zeta\in U(z_+, \delta),
\end{cases}
\end{equation}
where $R_{1}^{(m)}=Q_{1}^{(m)}$ for $m = -1, -2$ and $R_{1}^{(m)}=-Q_{1}^{(m)}$ for  $m=0,1,2,\cdots$. In particular, $R_{1}^{(-1)}$ is given explicitly in \eqref{eq:R1--1}. Moreover, since $c_{3}$ is bounded away from zero when $\mu$ lies in a fixed neighborhood of $(-\infty,M]$, all the coefficients $R_{1}^{(m)}$ are analytic functions of $\mu$ for $ \mu \in (-\infty,M]$. 

The case for $k >1$ is similar. First, it follows from the definition of $g(\zeta)$ in \eqref{eq-def-g} that 
\begin{equation}
\begin{split}
\left(\frac{3}{2}g(\zeta)\right)^{-2k}&=\sum\limits_{i=-3k}^{\infty}g_{2k,i}(\zeta-z_{+})^{i}, \\
\left(\frac{3}{2}g(\zeta)\right)^{-2k+1}&=(\zeta-z_{+})^{-\frac{1}{2}}\sum\limits_{i=-3k}^{\infty}g_{2k-1,i}(\zeta-z_{+})^{i},
\end{split}
\end{equation}
where the coefficients $g_{k,m}$ are all polynomials of $\frac{1}{c_3}$ with degree at least 2, and analytic functions of $\mu$ in the neighborhood of $(-\infty,M]$. Note that  $g_{2k-1,-3k}=0$. It then follows \eqref{eq-def-Qk} that $Q_{k}(\zeta)$ possesses a pole of order $n_{k}$ at $\zeta=z_{+}$. Using the above formula, we also have the Laurent series for $Q_{k}(\zeta)$ below 
\begin{equation}
Q_{k}(\zeta)=\sum\limits_{m=-n_{k}}^{+\infty}Q_{k}^{(m)}(\zeta-z_{+})^{m},\qquad |z-z_+| > 0, 
\end{equation}
with
\begin{equation}\label{eq-Qk-explicit-representation}
Q_{2k}^{(m)}=\left(\begin{matrix}
\hat{t}_{2k}g_{2k,m}&0\\ 0&t_{2k}g_{2k,m}
\end{matrix}\right), \quad Q_{2k-1}^{(m)}=\left(\begin{matrix}
0&\hat{t}_{2k-1}g_{2k-1,m+1}\\ t_{2k-1}g_{2k-1,m}&0
\end{matrix}\right),
\end{equation}
for all $m\geq -n_{k}$.  From \eqref{eq-Rk-reccursion}, we have 
\begin{equation}
R_{k}(\zeta)=\frac{1}{2\pi i}\oint_{\partial U(z_{+},\delta)}\frac{Q_{k}(s)+\sum\limits_{m=0}^{k-1}R_{m,-}(s)Q_{k-m}(s)}{s-\zeta}ds.
\end{equation}
A straightforward computation yields 
\begin{equation}
R_{k}(\zeta)=\begin{cases}\sum\limits_{m=1}^{n_{k}}\frac{R_{k}^{(-m)}}{(\zeta-z_{+})^{m}},& \zeta\in\mathbb{C}\setminus{\overline{U(z_{+},\delta)}}, \smallskip \\
\sum\limits_{m=0}^{+\infty}R_{k}^{(m)}(\zeta-z_{+})^{m},&\zeta\in U(z_{+},\delta),
\end{cases}
\end{equation}
with the coefficients given as
\begin{equation}\label{eq-Rk-recurrence}
R_{k}^{(m)}=Q_{k}^{(m)}+\sum\limits_{j=1}^{k-1}\left(\sum\limits_{i=0}^{p_{j}+m}R_{j}^{(i)}Q_{k-j}^{(-i+m)}\right), \quad k\geq 2,
\end{equation}
for all $m \geq -n_{k}$. As $g_{k,m}$ are all analytic functions of $\mu$ in the neighborhood of $(-\infty,M]$, by mathematical induction, we can conclude that  $R_{k}^{(m)}$ are all polynomials of $\frac{1}{c_{3}}$ and analytic for $ \mu \in (-\infty,M]$. In particular, we have
\begin{equation}
\begin{split}
R_{2}^{(-1)}&=Q_{2}^{(-1)}+R_{1}^{(0)}Q_{1}^{(-1)}+R_{1}^{(1)}Q_{1}^{(-2)}=Q_{2}^{(-1)}-Q_{1}^{(0)}Q_{1}^{(-1)}-Q_{1}^{(1)}Q_{1}^{(-2)},
\end{split}
\end{equation}
which gives us \eqref{eq:R2--1}. This finishes the proof of the lemma.
\end{proof}

\begin{remark}
Besides \eqref{eq:R1--1} and \eqref{eq:R2--1}, it is possible to get explicit formulas of $R_{k}^{(-1)}$ for all $k \geq  3$ with the aid of \eqref{eq-Qk-explicit-representation} and the recursion formula \eqref{eq-Rk-recurrence}. Moreover, using \eqref{eq-R-series} and \eqref{eq-Rk-Laurent-representation}, we get the following approximation for the $\frac{1}{\zeta}$-coefficient $C_1$ in \eqref{eq-R3-C1-T1}:
\begin{equation}\label{eq-C1-asym-expansion}
C_{1}(x,t)\sim\sum\limits_{k=1}^{\infty}\frac{R_{k}^{(-1)}}{|x|^{\frac{7k}{6}}}, \qquad \text{ as  } x\to+\infty.
\end{equation}
Due to  Lemma \ref{lem-Rk-properties}, one can see that $R_{2k}^{(-1)}$ and $R_{2k-1}^{(-1)}$ are diagonal and anti-diagonal matrices, respectively.
\end{remark}

\section{Proof of main results}\label{sec-proof-result}

In the last section, we first derive uniform asymptotics for the tritronqu\'{e}e solution $u(x,t)$ in Theorem \ref{Thm-asym-u}. Then, we establish the total integrals for $u(x,t)$ and the associated Hamiltonian $H_1(x,t)$ in Theorem \ref{Thm-total-integral}.

\medskip

\noindent\textbf{Proof of Theorem \ref{Thm-asym-u}:}  First, recalling \eqref{eq-u-by-T} and \eqref{eq-R3-C1-T1}, we have 
\begin{equation}
u(x,t)=2|x|^{\frac{1}{3}}\left(\frac{z_{+}}{4}+(C_{1})_{11}\right)-|x|^{\frac{1}{3}}(C_{1})_{12}^{2}.
\end{equation}
It then follows from \eqref{eq-C1-asym-expansion} and the above formula that
\begin{equation}\label{eq-asym-u-proof}
\begin{split}
u(x,t)\sim&2|x|^{\frac{1}{3}}\left(\frac{z_{+}}{4}+\sum\limits_{k=1}^{\infty}\frac{(R_{2k}^{(-1)})_{11}}{|x|^{\frac{7k}{3}}}\right)-|x|^{\frac{1}{3}}\left(\sum\limits_{k=1}^{\infty}\frac{(R_{2k-1}^{(-1)})_{12}}{|x|^{\frac{(2k-1)7}{6}}}\right)^2, \qquad \textrm{as } x\to+\infty, 
\end{split}
\end{equation}
This gives us the asymptotic expansion \eqref{eq-asym-u} with the coefficients given below
\begin{equation}\label{e_k-by-R}
e_{k}^+(\mu) := e_{k}(z_{+},\mu)=2(R_{2k}^{(-1)})_{11}-\sum\limits_{i=1}^{k}(R_{2i-1}^{(-1)})_{12}(R_{2(k-i)+1}^{(-1)})_{12},\quad k\geq 1.
\end{equation}
When $k=1$, with the explicit expressions of $R_{1}^{(-1)}$ and $R_{2}^{(-1)}$ in \eqref{eq:R1--1} and \eqref{eq:R2--1}, we obtain
\begin{equation}
e_{1}^+(\mu)=2(R_{2}^{(-1)})_{11}-(R_{1}^{(-1)})_{12}^{2}=-\frac{64}{3}(z_{+}^2-8\mu)^{-3}+\frac{256z_{+}^2}{3}(z_{+}^2-8\mu)^{-4}.
\end{equation} 

Next, we show that $e_{k}^+(\mu)$ are  uniformly bounded for $\mu\in(-\infty,M]$. To see this, we recall Lemma \ref{lem-Rk-analytic}, which indicates that all the coefficients $R_{k}^{(m)}$ with $k \in \mathbb{N^+}$ and $m=-n_{k},-n_{k}+1,\cdots$, are polynomials of $\frac{1}{c_{3}}$. In particular, $\deg R_{1}^{(-1)} \geq 1$ and $\deg R_{k}^{(-1)} \geq 2$ for $k \geq 2$. This implies that, there exists a constant $\tilde{C}$, such that
\begin{equation}
|R_{1}^{(-1)}|\leq \tilde{C}|(z_{+}^2-8\mu)^{-1}| \quad\text{and} \quad  |R_{k}^{(-1)}| \leq \tilde{C}(z_{+}^2-8\mu)^{-2},\quad k\geq 2,
\end{equation}
uniformly for all $\mu\in(-\infty,M]$. 
Substituting the above approximations into \eqref{e_k-by-R}, we can see that all $e_{k}^+(\mu)$ are bounded uniformly for all $\mu\in(-\infty,M]$. Moreover, they satisfy the approximation \eqref{eq:ek-large-mu} as $\mu \to -\infty$.

When $x\to-\infty$, the RH analysis is similar to the case when $x\to+\infty$. One only needs to replace $z_{+}$ by $z_{-}$; see the definitions of $z_\pm$ in \eqref{zpm-def}. We have the asymptotic expansion \eqref{eq-asym-u}, with $z_{+}$ replaced by $z_{-}$. The coefficients $e_{k}^{-}(\mu)$ share the same properties as  $e_{k}^+(\mu)$.  

This finishes the proof of the theorem. \qed

\bigskip

Next, we will establish the total integrals.
\medskip

\noindent\textbf{Proof of Theorem \ref{Thm-total-integral}:} When $t$ is fixed, in view of the asymptotics of $u(x,t)$ in \eqref{eq-asym-u-fixed-t}, we add the factor $6^{\frac{1}{3}}x^{\frac{1}{3}}+2\cdot 6^{-\frac{1}{3}}tx^{-\frac{1}{3}}$ on both sides of  \eqref{eq-differential-equality-1}. Then, integrating  with respect to $x$ from $X_{1}$ and $X_{2}$, we have 
\begin{equation}
\begin{split}
&\int_{X_{1}}^{X_{2}}\left(u(x,t)+\sqrt[3]{6}x^{\frac{1}{3}}+\frac{2t}{\sqrt[3]{6}}x^{-\frac{1}{3}}\right)dx\\
=&H_{1}(X_{2},t)+\frac{3}{4}X_{2}^{\frac{4}{3}}+3\cdot 6^{-\frac{1}{3}}tX_{2}^{\frac{2}{3}}-H_{1}(X_{1},t)-\frac{3}{4}X_{1}^{\frac{4}{3}} -3\cdot 6^{-\frac{1}{3}}tX_{1}^{\frac{2}{3}}.
\end{split}
\end{equation}
Taking the limits $X_{1}\to-\infty$ and $X_{2}\to+\infty$ and making use of the asymptotic behavior of $H_{1}(x,t)$ in \eqref{eq-asym-Hamiltonians-fixed-t}, we immediately get \eqref{eq-total-integral-u}.

Now, we turn to the total integral of $H_1(x,t)$, and consider $I(t)$ in \eqref{eq-def-I(t)} and its alternative form \eqref{eq-integral-with-t}. 
According to \cite[Remark 4.3]{Grava-Kapaev-Klein-2015}, one can see that, for any fixed $x$,
\begin{equation}
u(x,t)=\frac{x}{t}+\mathcal{O}(-t)^{-4}, \qquad \textrm{as } t\to-\infty.
\end{equation}
It then follows from \eqref{eq-represent-Hamiltonian-1} that $H_{1}(x,t)=\mathcal{O}(-t)^{-1}$ as $t\to-\infty$, hence $I(-\infty)=0$. 
For any fixed $t$, let us rewrite \eqref{eq-integral-with-t} as 
\begin{equation} \label{eq-I(t)-1}
I(t) = -I(X_{2}, t)+I(X_{1},t)-\frac{1}{24}\log{|T_{2}|}+\frac{1}{24}\log{|T_{1}|},
\end{equation}
where $T_{1}$ and $T_{2}$ are two negative constants chosen to satisfy the relation $T_{1}|X_{1}|^{-\frac{2}{3}}=T_{2}|X_{2}|^{-\frac{2}{3}}$, and 
\begin{eqnarray} 
{\tiny I(X_k,t) = \int_{T_{k}}^{t}\left[\frac{H_{2}(X_{k},\tau)}{3}+\frac{u_{x}(X_{k},\tau)}{12}\right]d\tau+\int_{-\infty}^{T_{k}}\left[\frac{H_{2}(X_{k},\tau)}{3}+\frac{u_{x}(X_{k},\tau)}{12}-\frac{1}{24\tau}\right]d\tau } 
\end{eqnarray}
with $k=1,2$. Note that, to ensure the convergence at $\tau = -\infty$, an additional factor $\frac{1}{24 \tau}$ is introduced in the above integrand.

Next, we study $I(X_{2}, t)$ and $I(X_{1}, t)$  as $X_{2}\to+\infty$ and $X_{1}\to-\infty$,  respectively. Recall the asymptotics of $u_x(x,t)$ and $H_2(x,t)$ in \eqref{eq-asym-ux-uniform} and \eqref{eq-asym-Hamiltonians-uniformly-t-2}, and the fact that both $z_\pm$ and $\mu$ depend on $t$. Let us slightly abuse the notation by setting $\mu=\tau|X_{2}|^{-\frac{2}{3}}$ and $z_{+}:=z_{+}(\tau|X_{2}|^{-\frac{2}{3}})$. As discussed in Remark \ref{rem-error-H}, the error terms in \eqref{eq-asym-ux-uniform} and \eqref{eq-asym-Hamiltonians-uniformly-t-2} are still $o(1)$ after integrating with respect to $t$ from $- \infty$ to any positive fixed constant. Then, we have 
\begin{equation} \label{eq:I2-formual1}
\begin{split}
I(X_{2}, t)
=&\frac{1}{3}\int_{-\infty}^{t}\left[\left(\frac{z_+^5}{320}+\frac{3z_+^2}{8}-\frac{\mu z_+^{3}}{8}\right)|X_{2}|^{\frac{5}{3}}\right]d\tau+\frac{1}{3}\int_{T_{2}}^{t}\left[\frac{z_+^2+8\mu}{(z_+^2-8\mu)^2}|X_{2}|^{-\frac{2}{3}}\right]d\tau\\
&+\frac{1}{3}\int_{-\infty}^{T_{2}}\left[\frac{z_+^2+8\mu}{(z_+^2-8\mu)^2}|X_{2}|^{-\frac{2}{3}}-\frac{1}{8\tau}\right]d\tau+o(1), \qquad \textrm{as } X_{2}\to+\infty.
\end{split}
\end{equation}
Introduce a change of variables $z=z_{+}(\tau|X_{2}|^{-\frac{2}{3}})$ in the above integral. Recalling  $\mu=\frac{z_+^2}{24}+\frac{2}{z_+}$ (cf. \eqref{zpm-def}), we have
\begin{equation}
\frac{d\tau}{dz} = \frac{d\tau}{d \mu} \cdot \frac{d \mu}{dz} = |X_2|^{\frac{2}{3}} \left(\frac{z}{12}-\frac{2}{z^2}\right) .
\end{equation}
Based on the first paragraph in the proof of Lemma \ref{lem-g}, we know $-480^{\frac{1}{3}}<z_{+}(\tau|X_{2}|^{-\frac{2}{3}})<0$ for all $\tau \in (-\infty, t)$. This gives us $\frac{d\tau}{dz} < 0$ for all $\tau \in (-\infty, t)$ and $z_{+}(\tau|X_{2}|^{-\frac{2}{3}}) \to 0-$ as $\tau \to -\infty$. Let us introduce two more notations $z_{0}^{+}=z_{+}(t|X_{2}|^{-\frac{2}{3}})$ and $\tilde{z}_{2}=z_{+}(T_{2}|X_{2}|^{-\frac{2}{3}})$. Then, the integral \eqref{eq:I2-formual1} becomes
\begin{equation}
\begin{split}
I(X_{2}, t)=&\frac{1}{3}|X_{2}|^{\frac{7}{3}}\int_{0}^{z_{0}^{+}}\left[\left(\frac{z^5}{320}+\frac{3z^2}{8}-\left(\frac{z^2}{24}+\frac{2}{z}\right)\frac{z^{3}}{8}\right)\right]\left(\frac{z}{12}-\frac{2}{z^2}\right)dz\\
&+\frac{1}{3}\int_{\tilde{z}_{2}}^{z_{0}^{+}}\left[\frac{z^2+8\mu}{8z(z^2-8\mu)}\right]dz+\frac{1}{4}\int_{0}^{\tilde{z}_{2}}\left[\frac{1}{(z^2-8\mu)\mu}\right]dz+o(1).
\end{split}
\end{equation}
Using $\mu=\frac{z_+^2}{24}+\frac{2}{z_+}$ again to replace the variable $\mu$ in the above formula, we get
\begin{align}
&I(X_{2}, t)
=\frac{1}{3}|X_{2}|^{\frac{7}{3}}\int_{0}^{z_{0}^{+}}\left[\left(\frac{z^5}{320}+\frac{3z^2}{8}-\left(\frac{z^2}{24}+\frac{2}{z}\right)\frac{z^{3}}{8}\right)\right]\left(\frac{z}{12}-\frac{2}{z^2}\right)dz \label{eq-I2-3integrals} \\
&\qquad +\frac{1}{3}\int_{\tilde{z}_{2}}^{z_0^+}\left[\frac{z^3+12}{4z(z^3-24)}\right]dz+\int_{0}^{\tilde{z}_{2}}\left[\frac{9z^2}{(z^3-24)(z^3+48)}\right]dz+o(1),  \quad \textrm{as } X_{2}\to+\infty. \nonumber
\end{align}
In a similar way, with  $ z_{0}^{-}=z_{-}(t|X_{1}|^{-\frac{2}{3}})$ and $\tilde{z}_{1}=z_{-}(T_{1}|X_{1}|^{-\frac{2}{3}})$, we have 
\begin{align}
& I(X_{1}, t)=\frac{1}{3}|X_{1}|^{\frac{7}{3}}\int_{0}^{z_{0}^{-}}\left[\left(\frac{z^5}{320}-\frac{3z^2}{8}-\left(\frac{z^2}{24}-\frac{2}{z}\right)\frac{z^{3}}{8}\right)\right]\left(\frac{z}{12}+\frac{2}{z^2}\right)dz \label{eq-I1-3integrals} \\
&\qquad +\frac{1}{3}\int_{\tilde{z}_{1}}^{z_0^-}\left[\frac{z^3-12}{4z(z^3+24)}\right]dz+\int_{0}^{\tilde{z}_{1}}\left[\frac{-9z^2}{(z^3+24)(z^3-48)}\right]dz+o(1) ,  \quad \textrm{as } X_{1}\to-\infty. \nonumber
\end{align}
Since $z_{+}(\mu)=-z_{-}(\mu)$ for any $\mu\in(-\infty,M]$, we get
\[ \tilde{z}_{1}=z_{-}(T_{1}|X_{1}|^{-\frac{2}{3}}) = - z_{+}(T_{2}|X_{2}|^{-\frac{2}{3}})=-\tilde{z}_{2}\] 
due to our choice  $T_{1}|X_{1}|^{-\frac{2}{3}}=T_{2}|X_{2}|^{-\frac{2}{3}}$. This means that the third integral in \eqref{eq-I2-3integrals} and \eqref{eq-I1-3integrals} are actually the same. In the meantime, the second integral in \eqref{eq-I2-3integrals} and \eqref{eq-I1-3integrals} only differ by a $o(1)$-factor when $|X_{1,2}| \to \infty$ because $z_0^\pm \sim \mp 2\cdot 6^{\frac{1}{3}} + o(1)$. Then, when considering $I(X_1,t ) - I(X_2, t)$, only the first integrals make a contribution. As a consequence, combining \eqref{eq-I(t)-1}, \eqref{eq-I2-3integrals} and \eqref{eq-I1-3integrals}, we have
\begin{equation}
\begin{split}
I(t)=&\frac{|X_{2}|^{\frac{7}{3}}}{3}\int_{0}^{z_{0}^{+}}\frac{(z^3+24)(z^3+60)}{5760}dz\\
&+\frac{|X_{1}|^{\frac{7}{3}}}{3}\int_{0}^{z_{0}^{-}}\frac{(z^3-24)(z^3-60)}{5760}dz+\frac{1}{36}\log{\frac{|X_{1}|}{|X_{2}|}}+o(1),
\end{split}\end{equation}
as $X_{1}\to-\infty$ and $X_{2}\to+\infty$. Obviously, the remaining integrals in the above formula can be computed explicitly. When $t$ is fixed, we further expand $z_{0}^{+}=z_{+}(t|X_{2}|^{-\frac{2}{3}})$ and $ z_{0}^{-}=z_{-}(t|X_{1}|^{-\frac{2}{3}})$ as $|X_{1,2}| \to \infty$ (cf. \eqref{eq-asymp-z-pm}) to obtain
\begin{equation}\label{eq-asym-I(t)}
\begin{split}
I(t)=&\frac{9}{28}|X_{1}|^{\frac{7}{3}}+\frac{9}{5}6^{-\frac{1}{3}}t|X_{1}|^{\frac{5}{3}}+t^2|X_{1}|+\frac{6^{\frac{1}{3}}t^3}{3}|X_{1}|^{\frac{1}{3}}+\frac{1}{36}\log{|X_{1}|}\\
&-\left(\frac{9}{28}|X_{2}|^{\frac{7}{3}}+\frac{9}{5}6^{-\frac{1}{3}}t|X_{2}|^{\frac{5}{3}}+t^2|X_{2}|+\frac{6^{\frac{1}{3}}t^3}{3}|X_{2}|^{\frac{1}{3}}+\frac{1}{36}\log{|X_{2}|}\right)+o(1)
\end{split}
\end{equation}
as $X_{1}\to-\infty$ and $X_{2}\to+\infty$.  The above approximation for the integral $I(t)$ does not really make sense unless we modify the integrand in \eqref{eq-def-I(t)} and replace $H_1(x,t)$ by $H_{1}(x,t)+\frac{3}{4}6^{\frac{1}{3}}x^{\frac{4}{3}}+3\cdot 6^{-\frac{1}{3}}tx^{\frac{2}{3}}+t^2+\frac{6^{\frac{1}{3}}t^3}{9}x^{-\frac{2}{3}}+\frac{x}{36(x^2+1)}$. More precisely, it follows from \eqref{eq-def-I(t)} and \eqref{eq-asym-I(t)} that
\begin{equation}
\int_{X_{1}}^{X_{2}}\left(H_{1}(x,t)+\frac{3}{4}6^{\frac{1}{3}}x^{\frac{4}{3}}+3\cdot 6^{-\frac{1}{3}}tx^{\frac{2}{3}}+t^2+\frac{6^{\frac{1}{3}}t^3}{9}x^{-\frac{2}{3}}+\frac{x}{36(x^2+1)}\right)dx=o(1),
\end{equation}
as $X_{1}\to-\infty$ and $X_{2}\to+\infty$. This formula gives us the desired total integral in \eqref{eq-total-integral-H1}, which finishes the proof of Theorem \ref{Thm-total-integral}. \qed

\section*{Acknowledgments}
Dan Dai was partially supported by a grant from the City University of Hong Kong (Project No. 7005597), and grants
from the Research Grants Council of the Hong Kong Special Administrative Region, China (Project No. CityU 11300520 and CityU 11311622). Wen-Gao Long was partially supported by the Natural Science Foundation of Hunan Province [Grant no. 2020JJ5152],
the General Project of Hunan Provincial Department of Education [Grant no. 19C0771],
and the Doctoral Startup Fund of Hunan University of Science and Technology [Grant no. E51871].

\end{document}